\def\version{01/09/2008}

\documentclass[a4paper,11pt]{amsart}
\usepackage{amssymb,amscd,amsxtra}
\usepackage{fullpage}
\usepackage[all]{xy}
\usepackage{hyperref}

\theoremstyle{plain}
\newtheorem{thm}{Theorem}[section]
\newtheorem{lem}[thm]{Lemma}
\newtheorem{prop}[thm]{Proposition}
\newtheorem{cor}[thm]{Corollary}

\theoremstyle{definition}
\newtheorem{rem}[thm]{Remark}

\numberwithin{equation}{section}

\def\ie{\emph{i.e.}}
\def\ds{\displaystyle}
\def\:{\colon}
\def\.{\cdot}

\def\<{\left\langle}
\def\>{\right\rangle}
\def\({\left(}
\def\){\right)}
\def\ph#1{\phantom{#1}}
\def\epsilon{\varepsilon}

\def\leq{\leqslant}
\def\geq{\geqslant}
\def\lla{\longleftarrow}

\def\lra{\longrightarrow}
\def\Lra{\Longrightarrow}
\def\ra{\rightarrow}
\def\hat#1{\widehat{#1}}

\def\iso{\cong}

\def\F{\mathbb{F}}

\def\k{\Bbbk}

\def\Z{\mathbb{Z}}

\def\ideal{\triangleleft}

\DeclareMathOperator*{\Gal}{Gal}

\DeclareMathOperator{\Tor}{Tor}

\def\id{\mathrm{id}}
\DeclareMathOperator*{\colim}{colim}

\DeclareMathOperator*{\holim}{holim}

\def\nr{\mathrm{nr}}

\def\Fc{\bar{\F}}
\def\Enr{E^\nr}
\def\Knr{K^\nr}

\def\h{h}
\def\phi{\varphi}
\newdir{ >}{{}*!/-5pt/@{>}}
\DeclareMathOperator{\ann}{ann}

\newdir{ >}{{}*!/-8pt/@{>}}

\begin{document}
\title{Galois extensions of Lubin-Tate spectra}
\author{Andrew Baker and Birgit Richter}
\address{Department of Mathematics, University of Oslo,
Norway.
\hfill\newline
{\ }Department of Mathematics, University of Glasgow,
Glasgow G12 8QW, Scotland. }
\email{a.baker@maths.gla.ac.uk}
\urladdr{http://www.maths.gla.ac.uk/$\sim$ajb}
\address{Department Mathematik der Universit\"at Hamburg,
20146 Hamburg, Germany.}
\email{richter@math.uni-hamburg.de}
\urladdr{http://www.math.uni-hamburg.de/home/richter/}
\subjclass[2000]{Primary 55P43, 13B05, 13K05; Secondary 55P60, 55N22}
\keywords{Galois extensions, separable closure, Witt vectors,
Lubin-Tate spectra}
\thanks{
\\
This paper is dedicated to Doug Ravenel and Steve Wilson who
have led algebraic topologists into brave new lands. \\
A.~Baker was partially supported by a YFF Norwegian
Research Council grant while visiting the University of Oslo.
We would like to thank John Rognes, Haynes Miller and the
referee for important comments on an early version, and Mark
Hovey for allowing us to include material appearing in
Section~\ref{sec:knlocal}. We also thank Bj\"orn Schuster
for sharing a stimulating bottle of wine in Oberwolfach.}
\date{\version}
\begin{abstract}
Let $E_n$ be the $n$-th Lubin-Tate spectrum at a prime $p$. There
is a commutative $S$-algebra $\Enr_n$ whose coefficients are built
from the coefficients of $E_n$ and contain all roots of unity whose
order is not divisible by~$p$. For odd primes~$p$ we show that
$\Enr_n$ does not have any non-trivial connected finite Galois
extensions and is thus separably closed in the sense of Rognes. At
the prime~$2$ we prove that there are no non-trivial connected Galois
extensions of $\Enr_n$ with Galois group a finite group $G$ with
cyclic quotient. Our results carry over to the $K(n)$-local context.
\end{abstract}

\maketitle

\section{Introduction}

For a prime~$p$, let $E_n$ be the Lubin-Tate spectrum whose coefficient
ring is
\[
\pi_*(E_n) = W\F_{p^n}[[u_1,\ldots,u_{n-1}]][u^{\pm 1}],
\]
where $u$ is an element of degree $-2$ and the $u_i$ have degree zero.
For a perfect field $\k$, $W\k$ denotes the ring of Witt vectors of
$\k$. The ring $W\F_{p^n}[[u_1,\ldots,u_{n-1}]]$ represents deformations
of the height~$n$ Honda formal group law over $\F_{p^n}$. The spectrum
$E_2$ features prominently in the work of Goerss, Henn, Mahowald and
Rezk~\cite{GHMR} on the calculation of the homotopy groups of the
$K(2)$-local sphere. Goerss, Hopkins and Miller~\cite{Re, GH} establish
an action of the extended Morava stabilizer group $\mathbb{G}_n$ by
$E_\infty$-maps on $E_n$. Work by Devinatz and Hopkins~\cite{DH} on
homotopy fixed point spectra identifies the $K(n)$-local sphere spectrum
$L_{K(n)}S$ as the homotopy fixed points of the action of $\mathbb{G}_n$
on $E_n$ and Rognes~\cite[section 5.4.1]{R} interprets the map
\[
L_{K(n)}S \simeq E_n^{\h \mathbb{G}_n} \lra E_n
\]
as a $K(n)$-local Galois extension with Galois group $\mathbb{G}_n$.

As observed in~\cite[section 5.4.1]{R}, there is a $K(n)$-local
Galois extension $E_n\lra\Enr_n$ obtained by adjoining all roots of
unity of order prime to~$p$ and then suitably completing the result,
so that
\[
\pi_*(\Enr_n) = W{\Fc_p}[[u_1,\ldots,u_{n-1}]][u^{\pm 1}],
\]
where $\Fc_p$ is the algebraic closure of $\F_{p^n}$. See
Section~\ref{sec:Ennr} for more details on $\Enr_n$.

Usually $\Enr_n$ is thought of as the maximal (abelian) unramified
extension of $E_n$ and our goal is to investigate the extent to which
it deserves this name. The coefficients do not allow for non-trivial
connected Galois extensions of graded commutative rings, and we will
show that there are no non-trivial connected finite Galois extensions
of $\Enr_n$ as a commutative $S$-algebra, at least if we work away
from the prime~$2$. Here, we use the notion of connectedness in the
sense of Rognes~\cite[10.2]{R}, thus a connected commutative $S$-algebra
is one without non-trivial idempotents. This is crucial because for
every commutative $S$-algebra $A$ we can always consider the trivial
$G$-Galois extension $A \lra \prod_G A$ for an arbitrary finite group~$G$.
We will recall some basic facts about connectedness in
Section~\ref{sec:Galoisthy}.

Our main result confirms Rognes'~\cite[conjecture 1.4]{R}.
\begin{thm}\label{thm:mainintro}
For an odd prime $p$, let $B/\Enr_n$ be a finite Galois extension
with non-trivial Galois group. Then $B$ is not connected. Hence
$\Enr_n$ is a  maximal connected Galois extension of $E_n$.
\end{thm}

For $p=2$ we show that any finite Galois extension $B/\Enr_n$ whose
Galois group has a cyclic quotient is not connected. At the moment
we are unable to prove that there are no non-trivial connected Galois
extensions of $\Enr_n$ at $p=2$ with a Galois group which has only
finite simple non-abelian quotients.

This result extends our earlier work of~\cite[example~42]{BR:Invt},
in which we showed that each abelian Galois extension $\Enr_n \lra B$
with Galois group whose order is prime to~$p$ gives rise to an algebraic
Galois extension $\pi_*(\Enr_n) \lra \pi_*(B)$, where the target is
concentrated in even degrees.

Rognes~\cite[definition 10.3.1]{R} calls a connected commutative
$S$-algebra $A$ \emph{separably closed} if there are no $G$-Galois
extensions $A \lra B$ with $G$ finite and non-trivial and $B$
connected, \ie, if each finite $G$-Galois extension $A \lra B$ has
a trivial Galois group or $B$ not connected.

In this terminology we prove that for an odd prime~$p$ the spectrum
$\Enr_n$ is separably closed. We conjecture that it is also separably
closed when $p=2$.

In section \ref{sec:knlocal} we will show that our results hold
$K(n)$-locally, \ie, that there are no non-trivial connected $K(n)$-local
Galois extensions with finite Galois group at odd primes and with
finite Galois group with cyclic quotient for the even prime.

So far, not many examples of separably closed commutative $S$-algebras
are known. In~\cite[theorem~10.3.3]{R}, Rognes proves that the
(unlocalized) sphere spectrum is separably closed. His proof uses
the fact that the ring of integers is separably closed,
see~\cite[proposition~10.3.2]{R}. We show that $\Enr_n$ is a separable
closure of the sphere in the $K(n)$-local category for all $n$ and
all odd primes~$p$. Hovey and Strickland showed that the $K(n)$-local
category is irreducible, \ie, it has no non-trivial localising (or
colocalising) subcategories~\cite[section~7]{HS}.

In~\cite{BR:Gal} we used the convention that for a Galois extension
of commutative $S$-algebras $A \lra B$ it is assumed that $B$ is
faithful as an $A$-module in the sense of~\cite[definition 4.3.1]{R}.
For the investigation of possible Galois extensions of $\Enr_n$ we
do not need this assumption because we can exploit the fact that
$\Enr_n$ has a residue field that is a finite cell $\Enr_n$-module
spectrum. We are grateful to John Rognes who suggested that line of
argument. Therefore a $G$-Galois extension of commutative $S$-algebras
$A \lra B$ is understood to consist of the following data
(compare~\cite[definition~4.1.3]{R} and~\cite[definition~1.4.4]{BR:Gal}).

Let $A$ be a commutative $S$-algebra and let $B$ be a commutative
cofibrant $A$-algebra. Let $G$ be a finite (discrete) group and
suppose that there is an action of $G$ on $B$ by commutative
$A$-algebra morphisms. Then $B/A$ is a $G$-Galois extension if
it satisfies the following two conditions:
\begin{itemize}
\item
The natural map $i\: A \lra B^{\h G} = F(EG_+,B)^G$ is a weak
equivalence of $A$-algebras.
\item
The canonical map of $B$-algebras
\[
h \: B \wedge_A B \lra F(G_+,B)
\]
that is induced from the action of $G$ on the right hand factor
of~$B$ is an equivalence.
\end{itemize}
If $i$ and $h$ are $X$-equivalences for some spectrum $X$, then
$B/A$ is called an $X$-local $G$-Galois extension.

\section{The spectrum $\Enr_n$}\label{sec:Ennr}

For ease of reference, we provide some details on the spectrum
$\Enr_n$, expanding on the discussion of~\cite[section 5.4.1]{R}.

For each $k\geq1$, using the methods of~\cite{SVW,BR:Gal}, it
follows that there is a Galois extension $W\F_{p^{nk}}E_n/E_n$
with Galois group $\Gal(\F_{p^{nk}}/\F_{p^n})\iso C_{k}$ and
\[
\pi_*(W\F_{p^{nk}}E_n) = W\F_{p^{nk}}\otimes_{W\F_{p^{n}}}\pi_*(E_n).
\]
Whenever $k \mid m$, there is an $E_n$-algebra morphism
$W\F_{p^{nk}}E_n \lra W\F_{p^{nm}}E_n$ which on homotopy groups
induces the obvious homomorphism obtained from the natural
inclusions $W\F_{p^{nk}} \lra W\F_{p^{nm}}$. Taking the colimit
in the category of commutative $E_n$-algebras leads to a spectrum
which is not $K(n)$-local, although each of the $W\F_{p^{nk}}E_n$
is $K(n)$-local since it is a finite wedge of copies of $E_n$.
The homotopy ring
\[
\pi_*(\colim_k W\F_{p^{nk}}E_n)
       \cong \colim_k W\F_{p^{nk}}[[u_1,\ldots,u_{n-1}]][u^{\pm 1}]
\]
is Noetherian, regular and local. Let $K_n$ denote the $E_n$-module
spectrum $E_n/\mathfrak{m}$ with $\mathfrak{m}$ denoting the maximal
ideal $(p,u_1,\ldots,u_{n-1})$ in $\pi_0(E_n)$. The finiteness of
$K'=\colim_k W\F_{p^{nk}}K_n$ over $E'=\colim_kW\F_{p^{nk}}E_n$
ensures that we can apply~\cite[theorem 6.4]{BL} and obtain that
\begin{equation*}
L^{E'}_{K'}E' = \holim_k E'/\mathfrak{m}^k
\end{equation*}
and
\begin{equation*}
\pi_*L^{E'}_{K'}E' = \lim_k (\pi_*E')/\mathfrak{m}^k
= W\Fc_p[[u_1,\ldots,u_{n-1}]][u^{\pm 1}].
\end{equation*}
Finally, a $\Gamma$-cohomology obstruction theory argument similar
to~\cite{RR,BR:Gamma} applies to show that $L^{E'}_{K'}E' = \Enr_n$
does indeed have a unique commutative $E_n$-algebra structure.

An alternative way to construct $\Enr_n$ is by considering the Honda
formal group law over $\Fc_p$ and its deformation theory with
respect to the complete local ring $W\Fc_p[[u_1,\ldots,u_{n-1}]]$.
Then Goerss-Hopkins-Miller obstruction theory~\cite[\S 7]{GH} shows
that $\Enr_n$ has a unique $E_\infty$-structure realizing
$(\Enr_n)_*(\Enr_n)$ as a commutative $\pi_*(\Enr_n)$-algebra and
also that $\Enr_n$ has an action of the group
$\mathbb{S}_n \rtimes \hat{\mathbb{Z}}$ via maps of $E_\infty$ ring
spectra. Rognes~\cite[5.4.6]{R} shows that
\begin{equation*}
L_{K(n)}S \lra \Enr_n
\end{equation*}
is a $K(n)$-local profinite Galois extension.

\section{Some results on Galois theory}\label{sec:Galoisthy}

We recall some facts about algebraic and topological Galois theory.
We begin with some algebraic results about Galois extensions of
graded commutative rings. In the following, $G$ will always be
a finite group.

Let $R\lra S$ be a $G$-Galois extension of graded commutative
rings. We include the following discussion along the lines
of~\cite[theorem~1.3]{CHR} at the suggestion of the referee,
because we do not know of any convenient source that states
that in the context of graded Galois extensions $S$ is finitely
generated projective over $R$. The impatient reader is invited
to move on directly to Proposition~\ref{prop:Galois-basechange}.

The unramified condition gives an isomorphism of
$S\otimes_R S$-modules
\[
h\:S\otimes_R S \xrightarrow{\iso} \prod_G S;
\quad
x\otimes y \mapsto (x g(y))_{g\in G},
\]
where the bimodule structure on the right hand side
is given by
\[
a(t_g)b = (at_g g(b))_{g\in G}.
\]
Then there is a map
\[
\sigma\: S \lra S\otimes_R S;
\quad
x\mapsto (x \otimes 1)h^{-1}(\delta_1),
\]
where for each $g\in G$, $\delta_g=(\delta_{g,h})_{h\in G}$
is the element of $\prod_G S$ which has zeros everywhere
except for a one in the entry corresponding to $g$. It is
easy to see that $\sigma$ is a bimodule map and when composed
with the product $\mu\:S\otimes_R S\lra S$ we obtain
$\mu\sigma = \id_S$. So $S$ is separable. In particular
there is an idempotent
\[
e = h^{-1}(\delta_1) \in S\otimes_R S,
\]
and we can write
\[
e = \sum_i u_i\otimes v_i
\]
for some finite collection of elements $u_i,v_i\in S$.
Notice that
\[
\delta_1 = h(e) = \biggl(\sum_iu_i g(v_i)\biggr)_{g\in G},
\]
so for each $g\in G$ we have
\[
\sum_iu_i g(v_i) =
\begin{cases}
1& \text{if $g=1$}, \\
0& \text{otherwise}.
\end{cases}
\]
Define the following $R$-linear maps:
\[
\alpha_i\:S\lra R;
\quad
x \mapsto \sum_{g\in G} g(v_i) g(x).
\]
Calculating in $S$ we have
\[
\sum_i u_i\alpha_i(x)
    = \sum_i u_i\sum_{g\in G} g(v_i) g(x)
    = \sum_{g\in G}\biggl(\sum_i u_i g(v_i)\biggr)g(x)
    = x.
\]
Now we use a well known characterisation of finitely
generated projective modules that applies as well in
the graded case.
\begin{lem}\label{lem:projchar}
Let $R$ be a graded commutative ring and let $M$ be a
graded $R$-module. Then $M$ is a finitely generated
projective module if and only if for some~$n$ there
are elements $b_1,\ldots,b_n\in M$ and $R$-linear maps
$\alpha_1,\ldots,\alpha_n\:M\lra R$ such that for
every $x\in M$,
\[
x = \sum_i\alpha_i(x)b_i.
\]
\end{lem}

Thus we obtain the following result.

\begin{lem} \label{lem:projective}
A $G$-Galois extension $R\lra S$ is a finitely generated
projective $R$-module.
\end{lem}

\begin{prop}\label{prop:Galois-basechange}
Let $S/R$ be a $G$-Galois extension of graded commutative
rings. Then for any graded commutative $R$-algebra $T$,
$T\otimes_R S/T$ is also a $G$-Galois extension. In
particular, if $I\ideal R$ is an ideal, $(S/SI)/(R/I)$
is a $G$-Galois extension.
\end{prop}
\begin{proof}
A proof in the ungraded case can be found for instance
in~\cite[lemma~1.7]{CHR} and with the help of
Lemma~\ref{lem:projective} it carries over to the graded
case.
\end{proof}

In the following we need to understand base-change properties
of topological Galois extensions.
\begin{prop}\label{prop:TopGalois-basechange}
Let $B/A$ be a $G$-Galois extension of commutative $S$-algebras.
Suppose that $A \lra C$ is a map of commutative $S$-algebras and
assume that $C$ is weakly equivalent to a retract of a finite cell
$A$-module spectrum. Then $C\wedge_A B/C$ is also a $G$-Galois
extension.
\end{prop}
\begin{proof}
See~\cite[lemma~7.1.3]{R}.
\end{proof}

We need to understand base changes as above along $C$, where $C$
is a residue field in the sense of~\cite{BR:Invt} and which also
happens to be an $A$-algebra. For instance, we could take
$A=\Enr_n$, the $n$-th Lubin-Tate spectrum, and $C=\Knr_n$, the
associated Morava $K$-theory with one of its strict multiplicative
structures described in~\cite{A}. However, in these cases $C$ is
not a commutative $S$-algebra, but for our purposes it suffices
that its coefficient ring is a graded commutative ring.

\begin{cor}\label{cor:TopGalois-basechange-noncomm-htpy}
Let $B$ be a cofibrant commutative $A$-algebra. Assume
$\pi_*(B)/\pi_*(A)$ is a $G$-Galois extension and $C$ is an
associative $A$-algebra whose coefficient ring $\pi_*(C)$ is
a graded commutative $\pi_*(A)$-algebra. Then
$\pi_*(C\wedge_A B)/\pi_*(C)$ is also a $G$-Galois extension.
\end{cor}
\begin{proof}
The assumption on the homotopy rings implies that $\pi_*(B)$
is a finitely generated projective $\pi_*(A)$-module. Therefore
the relevant K\"unneth spectral sequence of~\cite{EKMM} collapses
to give an isomorphism
\[
\pi_*(C\wedge_A B) \iso \pi_*(C) \otimes_{\pi_*(A)} \pi_*(B).
\]
Now the result follows from Proposition~\ref{prop:Galois-basechange}.
\end{proof}

Note that the realizability results of~\cite{BR:Gal} imply that
in the situation above the algebraic $G$-Galois extension
$\pi_*(B)/\pi_*(A)$ can be realized by a $G$-Galois extension
of commutative $S$-algebras $A \lra B'$ with $B' \simeq B$.

\begin{prop}\label{prop:TopGalois-basechange-field}
Let $B/A$ be a $G$-Galois extension of commutative $S$-algebras.
Let $C$ be an associative $A$-algebra that is a retract of a
finite cell $A$-module spectrum and for which $\pi_*(C)$ is
a graded field. Then $\pi_*(C\wedge_A B)/\pi_*(C)$ is an
algebraic $G$-Galois extension.
\end{prop}
\begin{proof}
The assumption that $C$ is a retract of a finite cell
$A$-module spectrum guarantees that the homotopy fixed
points $(C \wedge_A B)^{\h G}$ are weakly equivalent
to $C \wedge_A B^{\h G} \simeq C \wedge_A A \simeq C$
by~\cite[lemma~6.2.6]{R}. In particular, this shows that
$C \wedge_A B$ is not contractible.

The unramified condition follows from the evident chain of
isomorphisms
\begin{multline*}
\pi_*(C\wedge_A B)\otimes_{\pi_*(C)}\pi_*(C\wedge_A B)
             \xrightarrow{\iso}\pi_*(C\wedge_A B\wedge_A B) \\
             \xrightarrow{\iso}\pi_*(\prod_G C\wedge_A B)
             \xrightarrow{\iso}\prod_G\pi_*(C\wedge_A B).
\end{multline*}
We know that $\pi_*(C\wedge_A B)$ is a finite dimensional
$\pi_*(C)$-vector space. Let $k$ be the dimension of
$\pi_*(C\wedge_A B)$ over $\pi_*(C)$. The $G$-equivariant
isomorphism
\[
\pi_*(C\wedge_A B)\otimes_{\pi_*(C)}\pi_*(C\wedge_A B)
              \xrightarrow{\iso}\prod_G\pi_*(C\wedge_A B)
\]
maps $\pi_*(C\wedge_A B)\otimes_{\pi_*(C)}\pi_*(C\wedge_A B)^G$
isomorphically onto a subspace of
\[
\bigl(\prod_G\pi_*(C\wedge_A B)\bigr)^G \iso \pi_*(C\wedge_A B).
\]
Note that this shows that $k^2 = k|G|$ and thus $k$ is the cardinality
of the group. Now we know that
\[
\dim_{\pi_*(C)} \pi_*(C\wedge_A B)\otimes_{\pi_*(C)}\pi_*(C\wedge_A B)^G
               = k \dim_{\pi_*(C)} \pi_*(C\wedge_A B)^G \leq k
\]
and therefore $\pi_*(C\wedge_A B)^G$ is $1$-dimensional
over $\pi_*(C)$.
\end{proof}

In our work we will need a basic lemma on idempotents on Galois
extensions. For background on idempotents on commutative
$S$-algebras, see~\cite[10.2]{R}. We just recall some of the
main results.

Let $R$ be a commutative $S$-algebra; then we say that a commutative
$R$-algebra $A$ \emph{splits} if there is a weak equivalence of
commutative $R$-algebras $A \simeq A_1 \times A_2$ for some
commutative $R$-algebras $A_1,A_2$ which satisfy
$A_1 \not\simeq * \not\simeq A_2$, \ie, they are homotopically
non-trivial as commutative $R$-algebras. If $A$ admits no such
splitting it is said to be \emph{connected}, otherwise it is
\emph{non-connected}.

Let $\mathcal{E}(A)$ denote the mapping space of non-unital
commutative $A$-algebra endomorphisms of~$A$. Rognes shows
in~\cite[lemma~10.2.3]{R} that $A$ is connected if and only
if the map of spaces $\{0,1\} \lra \mathcal{E}(A)$ that takes
$0$ to the constant map and $1$ to the identity map is a weak
equivalence. Furthermore he proves in~\cite[proposition~10.2.2]{R}
that $\pi_0\mathcal{E}(A)$ corresponds to the idempotents of
the ring $\pi_0(A)$, thus $A$ is connected if and only if
$\pi_0(A)$ is connected in the sense of algebra.

\begin{lem}\label{lem:split}
Let $\iota\: A \lra B$ be a $G$-Galois extension of commutative
$S$-algebras. If $A$ splits as $A \simeq A_1 \times A_2$, then
as a commutative $A$-algebra, $B$ splits as
$B \simeq B_1 \times B_2$, such that there are compatible maps
of commutative $S$-algebras $A_1 \lra B_1$ and $A_2 \lra B_2$
which are $G$-Galois extensions.
\end{lem}
We are grateful to the referee who replaced our earlier clumsier
proof by the following straightforward line of argument.
\begin{proof}
Let $B_i = A_i \wedge_A B$ for $i=1,2$. Since $A_1$ and $A_2$
are retracts of a finite cell $A$-module, namely $A$ itself,
$A_i \lra B_i$ is a $G$-Galois extension for $i=1,2$ by
Proposition~\ref{prop:TopGalois-basechange}. The equivalence
of commutative $S$-algebras $A \ra A_1 \times A_2$ induces
the equivalence
\begin{equation*}
B \simeq A \wedge_A B \ra (A_1 \times A_2) \wedge_A B
\end{equation*}
and the latter is equivalent to
\[
(A_1 \vee A_2)\wedge_A B \simeq A_1 \wedge_A B \vee A_2 \wedge_A B
= B_1 \vee B_2 \simeq B_1 \times B_2.
\qedhere
\]
\end{proof}

\section{Calculations with residue fields}\label{sec:ResFld}

We recall from~\cite[\S 3]{BR:Invt} the notion of a residue
field for a commutative $S$-algebra $R$. Let $\mathfrak{m}$
be a maximal ideal in $\pi_*(R)$. If there is an $R$-module
spectrum $W$ for which the $\pi_*(R)$-module $\pi_*(W)$ is
isomorphic to $\pi_*(R)/\mathfrak{m}$, then we call $W$ a
residue field of $R$ with respect to $\mathfrak{m}$. Note
that in general no multiplicative structure on $W$ is assumed.

In the case of $\Enr_n$, there is a version of the Morava
$K$-theory spectrum $\Knr_n$ with coefficient ring
\[
\pi_*(\Knr_n) = \pi_*(\Enr_n)/\mathfrak{m} = \Fc_p[u,u^{-1}].
\]
Thus $\pi_*(\Knr_n)$ is the residue field of the local
ring $\pi_*(\Enr_n)$ and so $\Knr_n$ is a residue field
for $\Enr_n$. By work of~\cite{A}, $\Knr_n$ admits the
structure of an associative $\Enr_n$-algebra.

\medskip
\noindent
\textbf{Notation}.
To simplify notation, we set $E=\Enr_n$ and $K=\Knr_n$
from now on.

\bigskip

For two $E$-modules $M,N$, at least one of which is cofibrant,
there is a K\"unneth spectral sequence
\begin{equation}\label{eqn:KSS-Kn}
\mathrm{E}^2_{s,t} =
\Tor^{\pi_*(K)}_{s,t}
           (\pi_*(K \wedge_{E} M),\pi_*(K \wedge_{E}N))
                  \Lra \pi_*(K \wedge_{E} M \wedge_{E} N)
\end{equation}
which collapses to give a K\"unneth isomorphism of $\pi_*(K)$-modules
\begin{equation}\label{eqn:KForm-Kn}
\pi_*(K \wedge_{E} M \wedge_{E} N)
\iso
\pi_*(K \wedge_{E} M) \otimes_{\pi_*(K)} \pi_*(K \wedge_{E} N).
\end{equation}

Since these homotopy groups are $2$-periodic, we can view them
as $\Z/2$-graded modules. For a $K$-module $V$, $\pi_*(V)$ is
equivalent to a $\Z/2$-graded $\Fc_p$-vector space and then
we can consider the dimensions of the even and the odd parts
separately. For an $E$-module spectrum $M$ we set
\[
d_0=\dim_{\Fc_p}\pi_0(K\wedge_{E} M),
\quad
d_1=\dim_{\Fc_p}\pi_1(K\wedge_{E} M).
\]

\begin{lem}\label{lem:dim}
Suppose that an $E$-module spectrum $M$ satisfies
\begin{equation}\label{eqn:UnramCond}
M \wedge_{E} M \simeq \prod_X M
\end{equation}
for some finite set $X$ of cardinality $|X|=m$. If
$\pi_*(K \wedge_{E} M)$ is a non-trivial finite
dimensional $\pi_*(K)$-module, then the dimensions $d_0$
and $d_1$ satisfy one of the following conditions:
\begin{itemize}
\item
$d_1=0$ and $d_0=m$.
\item
$d_1\neq0$, $m$ is even and $d_0 = m/2 = d_1$.
\end{itemize}
In particular, if $m$ is odd, then we must have the first
condition.
\end{lem}
\begin{proof}
Using the K\"unneth formula based on $K$, we have
\[
\dim_{\Fc_p}\pi_0(K\wedge_{E} M\wedge_{E} M) = d_0^2+d_1^2,
\quad
\dim_{\Fc_p}\pi_1(K\wedge_{E} M\wedge_{E} M) = 2d_0d_1.
\]
On the other hand, by the assumed splitting of $M \wedge_{E} M$
we obtain the equations
\[
d_0^2+d_1^2 = m d_0, \quad 2d_0d_1 = m d_1.
\]
Using these we establish the result.
\end{proof}

\section{Separable closure property at odd primes}\label{sec:oddprimes}

For an odd prime $p$ we can prove a general result.
\begin{thm} \label{thm:oddprimes}
Let $G$ be an arbitrary finite group and $p$ an odd prime.
Then for every $G$-Galois extension $B$ of $E$ there is a
weak equivalence of commutative $E$-algebras
\begin{equation*}
B \simeq  \prod_G E.
\end{equation*}
\end{thm}

\begin{proof}
We know from Proposition~\ref{prop:TopGalois-basechange-field}
that the Galois extension $E \lra B$ gives rise to a $G$-Galois
extension $\pi_*(K) \lra \pi_*(K \wedge_E B)$ of graded rings,
in particular, $\pi_*(K \wedge_E B)$ is a graded separable
$\pi_*(K)$-algebra. DeMeyer and Ingraham
showed in~\cite[proposition~II.2.3]{DeMI} that for a separable
(ungraded) algebra $A$ over a commutative ring $R$ any
$R$-projective $A$-module $M$ is also $A$-projective. Their
proof translates to the graded setting without any changes,
thus any $\pi_*(K)$-projective $\pi_*(K \wedge_E B)$-module
is also $\pi_*(K \wedge_E B)$-projective. Assume that $z$ is
a non-trivial element in $\pi_1(K \wedge_E B)$. Then the cyclic
$\pi_*(K \wedge_E B)$-submodule
$\pi_*(K \wedge_E B)z\subseteq \pi_*(K \wedge_E B)$ is projective
and so the surjection
$\phi\:\pi_*(K \wedge_E B) \lra \pi_*(K \wedge_E B)z$ given
by $\phi(a)=az$ is split by a $\pi_*(K \wedge_E B)$-homomorphism
$\rho\:\pi_*(K \wedge_E B)z\lra \pi_*(K \wedge_E B)$ under
which $\rho(z) = 1+w$ for some $w$ satisfying $wz=0$. Using
$\pi_*(K \wedge_E B)$-linearity we obtain
\[
\rho(z^2) = z(1+w) = z+zw = z \neq0.
\]
But we are working in odd characteristic, hence $z^2=0$.
Therefore the odd part of $\pi_*(K \wedge_E B)$ has to vanish.

Let $K'$ be the $E$-module spectrum $E/(p,u_1,\ldots,u_{n-2})$.
We consider the long exact sequence corresponding to the cofibre
sequence
\[
K' \wedge_E B \xrightarrow{u_{n-1}} K' \wedge_E B \lra K \wedge_E B.
\]
As we know from above, $\pi_*(K \wedge_E B)$ is concentrated
in even degrees. Thus we obtain the exactness of
\begin{multline*}
0 \ra \pi_{2q}(K' \wedge_E B) \xrightarrow{u_{n-1}}
\pi_{2q}(K'\wedge_E B) \xrightarrow{j} \pi_{2q}(K \wedge_E B)
\xrightarrow{\delta} \pi_{2q-1}(K' \wedge_E B) \\
\xrightarrow{u_{n-1}} \pi_{2q-1}(K' \wedge_E B) \ra 0.
\end{multline*}

The associated Bockstein spectral sequence (see~\cite[5.9.9]{We})
has $B_*^0 =\pi_*(K\wedge_E B)=0$ in odd degrees and all differentials
have degree $-1$ and are therefore trivial. Denote $\pi_*(K'\wedge_E B)$
by $A_*^0$. Then we obtain a short exact sequence
\[
0\ra A^0_i/(u_{n-1}A^0_{i} + {}_{u_{n-1}^r}A^0_i)\xrightarrow{\bar{j^r}}B^r_i
            \xrightarrow{\delta^r} {}_{u_{n-1}}A^0_{i -1}
\cap u_{n-1}^rA^0_{i-1} \ra 0
\]
for every $r \geq 0$. As we have that the evenly graded part
of $A_*^0$ is $u_{n-1}$-torsion free and that
\begin{equation*}
B_*^0=B_*^1=\cdots=B_*^\infty,
\end{equation*}
this yields that for an even degree $i= 2q$ we have
\begin{equation*}
{}_{u_{n-1}}A_{2q-1}^0 = {}_{u_{n-1}}A_{2q-1}^0 \cap u_{n-1}^rA_{2q-1}^0
\end{equation*}
for all $r\geq 0$. As $\pi_*(K'\wedge_E B)$ is a finitely generated
$\pi_*(E)$-module, we can deduce that $\pi_{2q-1}(K' \wedge_E B)$
is $u_{n-1}$-torsion free. But then the multiplication with $u_{n-1}$
has to be an isomorphism on $\pi_{2q-1}(K' \wedge_E B)$ so we see
that  $\pi_{2q-1}(K' \wedge_E B)$ is actually trivial for all $q$
and
\[
\pi_0(K \wedge_E B) \cong \pi_0(K'\wedge_E B)/(u_{n-1}).
\]
Considering the spectra $E/(p,u_1,\ldots,u_j)$ for $j=n-3,\ldots,0$
(with $p=u_0$) in a similar fashion we obtain
\begin{equation} \label{eq:pi0KB}
\pi_0(K \wedge_E B) \cong \pi_0(B)/(p,u_1,\ldots,u_{n-1})
         \cong \pi_0(K) \otimes_{\pi_0(E)} \pi_0(B).
\end{equation}

The quotient $\pi_0(B)/\mathfrak{m}$ is a separable extension
of the separably closed field $\Fc_p$, so
by~\cite[corollary~II.2.4]{DeMI}, this extension has to split
as
\begin{equation*}
\pi_0(K \wedge_E B) \cong \prod_G \Fc_p.
\end{equation*}

As $B$ is a finite cell $E$-module we know that $\pi_0(B)$ is
a finitely generated module over the Noetherian ring $\pi_0(E)$
and the calculation above shows that
\begin{equation*}
\pi_0(K \wedge_E B) \cong \pi_0(B)/\mathfrak{m}.
\end{equation*}
By the lifting of idempotents result of~\cite[corollary~7.5]{E:CA}
for instance, the orthogonal idempotents that give rise to this
splitting lift to the $\pi_0(E)$-algebra $\pi_0(B)$ and we obtain
the desired splitting of $B$ into $G$ copies of $E$.
\end{proof}

\section{Galois groups with cyclic quotients}\label{sec:cyclic}

We will consider Galois extensions of $E$ with Galois groups
having finite cyclic quotients. We note that the result
concerning these extensions is valid for all primes.
\begin{thm}\label{thm:CycQuot}
Let $B/E$ be a $G$-Galois extension where $G$ is a finite
group with a cyclic quotient of prime order. Then $B$ is
non-connected.
\end{thm}

\begin{cor}\label{cor:CycQuot}
Every $G$-Galois extension $B$ of $E$ with finite solvable
Galois group $G$ is non-connected. In this sense, the
commutative $E_n$-algebra $E$ is a maximal connected
solvable Galois extension of $E_n$.
\end{cor}

Of course for odd primes this result is covered by
Theorem~\ref{thm:oddprimes}. We include a complete proof
for all primes, because a reduction to the case $p=2$ does
not yield a much shorter proof and we feel that the full
proof offers some insight.

\begin{proof}
By assumption, there is a normal subgroup $N\ideal G$ for
which one of the following holds:
\begin{enumerate}
\item[(A)]
$G/N \iso C_\ell$ with $\ell$ a prime different from $p$,
\item[(B)]
$G/N \iso C_p$.
\end{enumerate}

\noindent
Case (A):
We consider the extension $\pi_*(K) \lra \pi_*(K \wedge_E B^{\h N})$
in which $\pi_*(K \wedge_E B^{\h N})$ is a finite dimensional algebra
over the graded field $\pi_*(K)$ and is an  algebraic $C_\ell$-Galois
extension by Proposition~\ref{prop:TopGalois-basechange-field}.

The argument of~\cite[example 42]{BR:Invt} shows that the homotopy 
groups $\pi_*(K \wedge_E B^{\h N})$ have to be concentrated in even 
degrees: as $\pi_*(K \wedge_E B^{\h N})$ is a $C_\ell$-Galois
extension of $\pi_*K$, each homotopy group
$\pi_{2n+1}(K \wedge_E B^{\h N})$ is a $C_\ell$-representation and
it has a decomposition into character eigenspaces because $\ell\neq p$.
If there were odd-degree elements and if $p$ is an odd prime then
the map $x_1 \otimes x_2 \mapsto (x_1g(x_2))_{g\in G}$ would have
a non-trivial kernel, thus contradicting the unramified condition.
For $p=2$, every irreducible character has odd order, and thus
an odd degree element of the corresponding summand is nilpotent,
because some odd power lies in the invariant part which is trivial
in odd degrees. If $x$ is such an element with $x^j=0$, then
$x^{j-1} \otimes x$ would be in the kernel of the above mentioned
map, because $x^{j-1}g(x) = \lambda_g x^{j} = 0$, where
$g(x) = \lambda_gx$.

Thus we can focus on the extension
$\pi_0(K \wedge_E B^{\h N})/\pi_0(K)$. Now $\pi_0(K) = \Fc_p$
is a separably closed field, so this extension has to split
completely,
\ie,
\[
\pi_0(K \wedge_E B^{\h N}) \cong \prod_{C_\ell}\Fc_p.
\]

We know that $\pi_0(B^{\h N})$ is a finitely generated
$\pi_0(E)$-module, and $\pi_0(E)$ is complete with respect
to the maximal ideal $\mathfrak{m} \ideal \pi_0(E)$. A
Bockstein spectral sequence argument similar to the proof
of Theorem~\ref{thm:oddprimes} shows that
$\pi_0(K \wedge_E B^{\h N}) \cong \pi_0(B^{\h N})/\mathfrak{m}$.
Therefore by the usual lifting of idempotents result
of~\cite[corollary~7.5]{E:CA}, the $\ell$ orthogonal idempotents
that give rise to this splitting lift to the $\pi_0(E)$-algebra
$\pi_0(B^{\h N})$. Thus we obtain a splitting of $\pi_0(E)$-algebras
\[
\pi_0(B^{\h N}) \iso \prod_{C_\ell}\pi_0(E).
\]
and accordingly
\begin{equation*}
B^{\h N} \simeq \prod_{C_\ell} E
\end{equation*}
so we see that $B^{hN}$ splits completely. Now by
Lemma~\ref{lem:split}, in Case~(A) we obtain a non-trivial
splitting of~$B$.

\smallskip
\noindent
Case (B):
When $G/N \iso C_p$, we consider two cases.

First let us assume that
\[
\dim_{\Fc_p}\pi_0(K\wedge_E B^{\h N}) = p,
\quad
\dim_{\Fc_p}\pi_1(K\wedge_E B^{\h N}) = 0.
\]
By Proposition~\ref{prop:TopGalois-basechange-field}, the
pair $\pi_0(K\wedge_E B^{\h N})/\pi_0(K)$ forms an algebraic
$C_p$-Galois extension. As in Case (A), we can now deduce
that this extension splits, so
\[
\pi_0(K\wedge_E B^{\h N}) \iso \prod_{C_p}\pi_0(K).
\]
From this isomorphism we obtain $p$ orthogonal idempotents
in $\pi_0(K\wedge_E B^{\h N})$, each realised by a map
\[
K \lra K\wedge_E B^{\h N}.
\]
Again we can lift the idempotents that cause this splitting
because of the completeness of $\pi_0(E)$ and therefore we
can realise the corresponding splitting as
\[
B^{\h N}\simeq \prod_G E.
\]

For $p=2$, according to Lemma \ref{lem:dim} we have to exclude
the possibility that
\[
\dim_{\Fc_p}\pi_0(K\wedge_E B^{\h N}) = 1,
\quad
\dim_{\Fc_p}\pi_1(K\wedge_E B^{\h N}) = 1.
\]
As the $C_2$-action preserves degree, we know that
\[
\dim_{\Fc_p}\pi_0(K\wedge_E B^{\h N})^{C_2} = 1,
\quad
\dim_{\Fc_p}\pi_1(K\wedge_E B^{\h N})^{C_2} = 0
\]
because $\pi_*(K)$ is concentrated in even degrees. Thus the
$C_2$-action on $\pi_0(K\wedge_E B^{\h N})$ has fixed points,
whereas the action on $\pi_1(K\wedge_E B^{\h N})$ must have
no non-trivial fixed points.

To finish, we adapt the argument in the proof
of~\cite[proposition~17]{SS} to show that
$\pi_1(K\wedge_E B^{\h N})$ must be trivial: If a finite $p$-group
$P$ acts on an arbitrary abelian $p$-torsion group $M$, then for
any non-zero element $x\in M$, the subgroup of $M$ generated by
the $P$-orbit of~$x$ is a finite $p$-group which is also a
$P$-submodule and so by~\cite[proposition~17]{SS} it has non-trivial
fixed points.

\medskip
Thus in each of Cases (A) and (B), $B^{\h N}$ is not connected
and therefore by using Lemma~\ref{lem:split}, we see that~$B$
as an $N$-Galois extension of $B^{\h N}$ is not connected either.
\end{proof}

\section{The $K(n)$-local case} \label{sec:knlocal}

Again we let $(E, K)$ be the pair $(\Enr_n, \Knr_n)$, but note
that the discussion in this section carries over to the pairs
$(E_n,K_n)$ and $(\widehat{E(n)}, K(n))$ as well. Let
$\mathcal{D}_E$ denote the stable homotopy category of $E$-module
spectra
and let $\mathcal{D}_{E,K}$ be the full subcategory generated by
$K$-local $E$-modules.
We denote the localization functor from $\mathcal{D}_{E}$ to
$\mathcal{D}_{E,K}$ by $L = L_{E,K}$.

\begin{rem} \label{rem:SorE}
We note that by~\cite[proposition~2.2]{Ho:colim}, for
$X\in\mathcal{D}_E$ (which gives rise to an element
$X\in\mathcal{D}_S$ by restriction of scalars)
\[
X\in\mathcal{D}_{E,K} \quad\Longleftrightarrow\quad
                                   X\in\mathcal{D}_{S,K}.
\]
This implies that for an $E$-module spectrum $X$ the two conditions
$K \wedge X \simeq *$ and $K \wedge_E X \simeq *$ are equivalent.
\end{rem}

The aim of the following is to provide a reference for the fact
that dualizable objects in $\mathcal{D}_{E,K}$ are retracts of
finite cell $E$-modules. The argument we present here is due
to Mark Hovey and we are grateful to him for allowing us to
include it here. We recall from~\cite[definition~1.5]{HS} the
definitions of dualizable and $F$-small.

We note that $K \wedge_E X$ is in $\mathcal{D}_{E,K}$, because
it is a $K$-module spectrum and therefore it is $K$-local. We
also know that $K$ is small in $\mathcal{D}_{E,K}$ since it is
a finite cell $E$-module. This can be seen by expressing it as
$K = E/\mathfrak{m}$ with $\mathfrak{m}=(p,u_1,\ldots,u_{n-1})$,
where the generating sequence of $\mathfrak{m}$ is regular. More
generally there are finite cell $E$-module spectra $E/\mathfrak{m}^s$
($s\geq 1$) which fit together to form the $\mathfrak{m}$-adic tower
\begin{equation}\label{eqn:m-adictower}
K = E/\mathfrak{m} \lla E/\mathfrak{m}^2 \lla E/\mathfrak{m}^3
                                                     \lla \cdots
\end{equation}
constructed in full generality in~\cite{BL}.
By \cite[theorem~1.1]{Wu}, this is can be constructed as a tower
of associative $E$-algebras. The fibre of the map
$E/\mathfrak{m}^{s+1}\lra E/\mathfrak{m}^s$ is a finite wedge of
suspensions of $K$, and this can be used to show that the Bousfield
classes of $E/\mathfrak{m}^s$ and $K$ coincide.

We also remark that if $X$ is a retract of a finite cell $E$-module,
then $\pi_*(X)$ is a finitely generated $\pi_*(E)$-module since
$\pi_*(E)$ is Noetherian; the converse also holds since $\pi_*(E)$
is a regular local ring.

\begin{lem}\label{lem:E-localfiniteness}
If $X$ is dualizable in $\mathcal{D}_{E,K}$, then for an arbitrary
indexing set~$A$ and $Y_\alpha \in \mathcal{D}_{E,K}$, the natural
map
\begin{equation} \label{eq:fsmall}
L\bigvee_{\alpha \in A} F_E(X, Y_\alpha)
     \lra F_E(X,L\bigvee_{\alpha \in A}Y_\alpha)
\end{equation}
is an isomorphism, \ie, $X$ is $F$-small in $\mathcal{D}_{E,K}$
in the sense of \cite[definition~1.5]{HS}.
\end{lem}
\begin{proof}
As $X$ is dualizable in $\mathcal{D}_{E,K}$, we obtain
\begin{align*}
F_E\biggl(X,L \bigvee_{\alpha \in A} Y_\alpha\biggr)
 & \cong
 L\biggl(F_E(X,E)\wedge_E L\bigvee_{\alpha\in A}Y_\alpha\biggr) \\
  &\cong L\bigvee_{\alpha \in A} (F_E(X,E)\wedge_E Y_\alpha) \\
  &\cong L\bigvee_{\alpha \in A} F_E(X,Y_\alpha).
\qedhere
\end{align*}
\end{proof}

\begin{lem}\label{lem:F-small}
If $X$ is $F$-small in $\mathcal{D}_{E,K}$, then $K \wedge_E X$
is also $F$-small in $\mathcal{D}_{E,K}$.
\end{lem}
\begin{proof}
The standard adjunction yields
\[
[K \wedge_E X, L\bigvee_{\alpha \in A} Y_\alpha]
             =  [K, F_E(X,L\bigvee_{\alpha \in A} Y_\alpha)].
\]
As $X$ is $F$-small this coincides with
$[K,L\bigvee_{\alpha \in A}F_E(X,Y_\alpha)]$ and the smallness
of $K$ turns this into
\[
\bigoplus_{\alpha \in A}[K,F_E(X,Y_\alpha)]
 \cong \bigoplus_{\alpha \in A}[K \wedge_E X, Y_\alpha].
\qedhere
\]
\end{proof}

\begin{lem}\label{lem:fdimvs}
If $K \wedge_E X$ is small in $\mathcal{D}_{E,K}$, then
$\pi_*(K \wedge_E X)$ is a finite-dimensional $\pi_*(K)$-vector
space.
\end{lem}
\begin{proof}
As $K$ is a field spectrum $K \wedge_E X$ splits as
\[
K \wedge_E X \cong \bigvee_i \Sigma^{\nu_i} K.
\]
Thus, if $K \wedge_E X$ is small, the isomorphism
\[
K \wedge_E X \lra \bigvee_i \Sigma^{\nu_i} K
                 = L\biggl(\bigvee_i\Sigma^{\nu_i} K\biggr)
\]
factors through a finite subwedge and hence the wedge must
be finite.
\end{proof}

\begin{lem}\label{lem:homlim}
If $X \in \mathcal{D}_{E}$, then $LX$ is isomorphic to the
homotopy limit
\[
\ds\holim_s(E/\mathfrak{m}^s \wedge_E X)
\]
taken in the category of $E$-module spectra.
\end{lem}

\begin{proof}
Each term $E/\mathfrak{m}^s \wedge_E X$ is an
$E/\mathfrak{m}^s$-module spectrum and therefore it is
$E/\mathfrak{m}^s$-local which is equivalent to being $K$-local.
As we know that the Bousfield classes of $E/\mathfrak{m}^s$
and $K=E/\mathfrak{m}$ coincide for all $s\geq 1$, the homotopy
limit\/ $\ds\holim_s(E/\mathfrak{m}^s \wedge_E X)$
is $K$-local.

We know that $K \wedge_E X$ is already $K$-local and the image
of the reduction homomorphism
\begin{equation*}
\pi_*(K\wedge_E E/\mathfrak{m}^{s+1})
                       \lra\pi_*(K\wedge_E E/\mathfrak{m}^s)
\end{equation*}
is $\pi_*(K)$ by~\cite[corollary~5.11]{BL}, hence using the
commutative diagram coming from the K\"unneth isomorphism
\[
\xymatrix{
\pi_*(K \wedge_E E/\mathfrak{m}^{s+1} \wedge_E X)
                        \ar[r]^{\iso\ph{abcdefgh}}\ar[d]
& \pi_*(K \wedge_E E/\mathfrak{m}^{s+1})\otimes_{\pi_*(K)}\pi_*(K
 \wedge_E X) \ar[d] \\
\pi_*(K \wedge_E E/\mathfrak{m}^s \wedge_E X) \ar[r]^{\iso\ph{abcdefgh}}
& \pi_*(K \wedge_E E/\mathfrak{m}^s)\otimes_{\pi_*(K)}\pi_*(K \wedge_E X)
}
\]
we find that
\begin{equation*}
\pi_*(\holim(E/\mathfrak{m}^n \wedge_E K \wedge_E X))
                              \cong \pi_*(K \wedge_E X).
\end{equation*}
Therefore the map
$\ds X \lra \holim_s(E/\mathfrak{m}^s \wedge_E X)$ is a
$K$-equivalence because
\begin{align*}
K \wedge_E \holim_s(E/\mathfrak{m}^s \wedge_E X)
      &\cong \holim_s(K \wedge_E E/\mathfrak{m}^s \wedge_E X) \\
      & \cong K \wedge_E X.
\qedhere
\end{align*}
\end{proof}
\begin{lem}\label{lem:L-complete}
If $X$ is in $\mathcal{D}_{E,K}$, then its homotopy $\pi_*(X)$
is an $L$-complete $\pi_*(E)$-module in the sense
of\/~\cite[definition~A.5]{HS}.
\end{lem}
\begin{proof}
Each term $\pi_*(E/\mathfrak{m}^s \wedge_E X)$ is bounded
$\mathfrak{m}$-torsion and hence it is $L$-complete. But
the class of $L$-complete modules is closed under limits,
$\lim^1$-terms and extensions~\cite[theorem~A.6]{HS}, and
therefore the Milnor sequence yields the result.
\end{proof}

\begin{lem}\label{lem:fingen}
If $X\in \mathcal{D}_{E,K}$ and $\pi_*(K \wedge_E X)$ is
finite-dimensional over $\pi_*K$, then $\pi_*(X)$ is finitely
generated over $\pi_*(E)$. In particular, $X$ is a retract
of a finite cell $E$-module.
\end{lem}
\begin{proof}
We prove that $\pi_*(E/(p,u_1,\ldots,u_i) \wedge_E X)$ is
finitely generated over $\pi_*(E)$ by downward induction
on~$i$. The case $i=n-1$ is guaranteed by the assumption
because $\pi_*(K)$ is finitely generated over $\pi_*(E)$.

We set
\[
X/(p,u_1,\ldots, u_i) = E/(p,u_1,\ldots,u_i) \wedge_E X.
\]
Then there is a long exact sequence of homotopy groups
induced by the cofibre sequence
\[
X/(p,u_1,\ldots, u_{i-1}) \xrightarrow{u_i}
X/(p,u_1,\ldots, u_{i-1}) \lra  X/(p,u_1,\ldots, u_{i}).
\]

Denoting the annihilator of $u_i$ in $\pi_*(X/(p,u_1,\ldots,u_{i-1}))$
by
\[
\ann(u_i,\pi_*(X/(p,u_1,\ldots,u_{i-1})),
\]
the long exact sequence yields the short exact sequence
\begin{multline*}
0 \ra \pi_*(X/(p,u_1,\ldots,u_{i-1}))/u_i
       \lra \pi_*(X/(p,u_1,\ldots,u_{i})) \\
       \lra \ann(u_i,\pi_*(X/(p,u_1,\ldots,u_{i-1})) \ra 0.
\end{multline*}
Since it injects into $\pi_*(X/(p,u_1,\ldots,u_{i}))$,
$\pi_*(X/(p,u_1,\ldots,u_{i-1}))/u_i$ is finitely generated. Hence
there is a free $\pi_*(E)$-module $F_*$ together with a map
\[
f\:F_* \lra \pi_*(X/(p,u_1,\ldots,u_{i-1}))
\]
such that the induced map
\[
F_* \lra \pi_*(X/(p,u_1,\ldots,u_{i-1}))/u_i
\]
is surjective. Let $N$ be the cokernel of $f$. The diagram
of exact sequences
\[
\xymatrix{
{F_*} \ar[r]^(0.3){f}\ar[d]_{u_i} & {\pi_*(X/(p,u_1,\ldots,u_{i-1}))}
\ar[d]_{u_i} \ar[r] & {N} \ar[d]_{u_i}\\
{F_*} \ar[r]^(0.3){f} \ar[d] & {\pi_*(X/(p,u_1,\ldots,u_{i-1}))} \ar[d]
 \ar[r] & {N} \ar[d] \\
{F_*/u_i} \ar[r]^(0.3){\bar{f}} & {\pi_*(X/(p,u_1,\ldots,u_{i-1}))/u_i}
\ar[r] & {0}
}
\]
tells us that $u_iN = N$ and $L$-completeness of $N$ implies that
$N=0$. Hence the map $F \lra \pi_*(X/(p,u_1,\ldots,u_{i-1}))$ is
surjective and so $\pi_*(X/(p,u_1,\ldots,u_{i-1}))$ is finitely
generated over $\pi_*(E)$ for all $0 \leq i \leq n-1$.

As $\pi_*(E)$ is a regular local ring this implies that $X$
is a retract of a finite cell $E$-module.
\end{proof}

\begin{lem} \label{lem:kgalois}
Let $G$ be a finite group. If $E \lra B$ is a $K(n)$-local
$G$-Galois extension, then
\begin{equation*}
\pi_*(K) \lra \pi_*(K \wedge_E B)
\end{equation*}
is a $G$-Galois extension.
\end{lem}
\begin{proof}
Let $C_h$ be the cofibre of $h\: B \wedge_E B \lra \prod_G B$
in the category of $E$-module spectra. We know that
$K \wedge C_h \simeq*$ and with the facts mentioned in
Remark~\ref{rem:SorE} we see  that this is equivalent to
$K \wedge_E C_h \simeq *$. The argument for the map
$i \: E \lra B^{hG}$ is similar. With these two facts at
hand we can mimic the proof of
Proposition~\ref{prop:TopGalois-basechange-field} to obtain
the result.
\end{proof}
\begin{thm}\label{thm:oddprimes-K(n)-local}
\begin{itemize}
\item[]
\item[(a)]
For an odd prime $p$ and a finite group $G$, every $K(n)$-local
$G$-Galois extension of $E$ is non-connected.
\item[(b)]
For $p=2$ and $G$ a finite group possessing a cyclic quotient,
every $K(n)$-local $G$-Galois extension of $E$ is non-connected.
\end{itemize}
\end{thm}
\begin{proof}
From Lemma~\ref{lem:kgalois} we know that
$\pi_*(K) \lra \pi_*(K \wedge_E B)$ is a $G$-Galois extension.
As a Galois extension $B$ is dualizable in $\mathcal{D}_{E,K}$
and the finiteness discussion above ensures that it is in fact
a retract of a finite cell-$E$-module spectrum. We can therefore
transfer our proofs to the $K(n)$-local setting, to show that
$\pi_0(K) \lra \pi_0(K \wedge_E B)$ splits and we can lift the
corresponding idempotents to $\pi_0(E) \lra \pi_0(B)$.
\end{proof}

\end{document}